\documentclass[11pt,twoside]{amsart}

\usepackage{amsmath}
\usepackage{amsthm}
\usepackage{amsfonts, amssymb}
\usepackage{mathrsfs}
\usepackage[all]{xy}
\usepackage{url}

\usepackage{latexsym}

\usepackage{graphicx}
\usepackage{graphics}
\usepackage[hang,footnotesize,bf]{caption}
\usepackage{float}

\theoremstyle{plain}
\newtheorem{lema}{Lemma}[section]
\newtheorem{prop}[lema]{Proposition}
\newtheorem{teo}[lema]{Theorem}

\newtheorem{coro}[lema]{Corollary}
\theoremstyle{remark}

\newtheorem{obs}[lema]{Remark}

\theoremstyle{definition}
\newtheorem{defi}[lema]{Definition}

\newtheorem{ejs}[lema]{Examples}

\newcommand{\bd}{\partial}
\newcommand{\pbd}{\tilde\partial}

\begin{document}

\title[Non-homogeneous Combinatorial Manifolds]{Non-homogeneous Combinatorial Manifolds}

\author[N.A. Capitelli]{Nicolas Ariel Capitelli}
\author[E.G. Minian]{Elias Gabriel Minian}
\thanks{The authors' research is partially supported by Conicet.}

\address{Departamento  de Matem\'atica-IMAS\\
 FCEyN, Universidad de Buenos Aires\\ Buenos
Aires, Argentina}

\email{ncapitel@dm.uba.ar} \email{gminian@dm.uba.ar}

\begin{abstract}
In this paper we extend the classical theory of combinatorial manifolds to the non-homogeneous setting. $NH$-manifolds are polyhedra which are locally like Euclidean spaces of varying dimensions.  We show that many of the properties of classical manifolds remain valid in this wider context. $NH$-manifolds appear naturally when studying Pachner moves on (classical) manifolds. We introduce the notion of $NH$-factorization and prove that $PL$-homeomorphic manifolds are related by a finite sequence of $NH$-factorizations involving $NH$-manifolds.
\end{abstract}



\maketitle

\section{Introduction}
The notion of manifold (piecewise linear, topological, differentiable) is 
central in mathematics. An $n$-manifold is an object which is locally like 
the Euclidean space $\mathbb{R}^n$. Concretely, in the piecewise linear setting a 
PL-manifold of dimension $n$ is a polyhedron in which every point has a 
(closed) neighborhood which is a PL-ball of dimension $n$. 

The theory of combinatorial manifolds (which are the triangulations of PL-manifolds) has been widely developed during the last ninety years. J.W. 
Alexander's Theorem on regular expansions, Newman's result on the 
complement of an $n$-ball in an $n$-sphere, Whitehead's Regular 
Neighborhood theory and the s-cobordism theorem are some of its most important
advances \cite{Gla, Hud, Lic, Rou}. More recently Pachner \cite{Pach} 
studied a set of elementary combinatorial operations or \textit{moves}, and showed that any 
combinatorial manifold can be transformed into any other PL-homeomorphic 
one by using these moves (see also \cite{Lic}).

It is well known that any combinatorial $n$-manifold is a homogeneous (or 
pure) simplicial complex, which means that all the maximal simplices have 
the same dimension. It is natural to ask whether it is possible to extend 
the theory of combinatorial manifolds to the non-homogeneous context. More 
concretely, the main goal of this article is to investigate the properties of those polyhedra which 
are locally like Euclidean spaces of varying dimensions (see Figures \ref{fig:ejemplos_nh_variedades_locales} and \ref{fig:ejemplos_nh_variedades}
below). In this paper we introduce the theory of \textit{non-homogeneous manifolds} or $NH$-manifolds, for short. We will show that many 
of the basic properties of (classical) manifolds are also satisfied in this 
much wider setting. 

We investigate shellability in the non-homogeneous context. It is well-known that any shellable complex is homotopy equivalent to a wedge of spheres and that the only shellable manifolds are balls and spheres (see \cite{Bjo} and \cite{Koz}). We prove that every shellable $NH$-manifold is in particular an $NH$-bouquet, which extends the classical result for manifolds. We also study the notion of regular expansion for $NH$-manifolds and prove a generalization of Alexander's Theorem.

Non-homogeneous manifolds appear naturally  when studying Pachner moves between manifolds. We introduce the notion of $NH$-factorization and prove that any two PL-homeomorphic manifolds (with or without boundary) are related by a finite sequence of factorizations involving $NH$-manifolds. When the manifolds are closed, the converse also holds.

\section{Preliminaries}



We start by fixing some notation and terminology. In this paper, all the simplicial complexes that we deal with are assumed to be finite. If a  simplex $\sigma$ is a face of a simplex $\tau$, we will write $\sigma<\tau$ and when $\sigma$ is an immediate face we write $\sigma\prec\tau$. A \emph{principal} or \emph{maximal} simplex in $K$ is a simplex which is not a proper face of any other simplex of $K$ and a \emph{ridge} in $K$ is an immediate face of a maximal simplex. A complex is said to be \emph{homogeneous} of dimension $n$ if all of its principal simplices have dimension $n$. The boundary  $\bd{K}$ of an $n$-homogeneous complex $K$  is the subcomplex generated by the mod $2$ union of the ($n-1$)-simplices. The set of vertices of a complex $K$ will be denoted by $V_K$.

The join of two simplices  $\sigma,\tau$ with $\sigma\cap\tau=\emptyset$ will be denoted by $\sigma\ast\tau$. Also $K\ast L$ will denote the join of the complexes $K$ and $L$. Given a simplex $\sigma\in K$,  $lk(\sigma,K)$ will denote its \emph{link}, which is the subcomplex $lk(\sigma,K)=\{\tau\in K:\ \tau\cap\sigma=\emptyset,\ \tau\ast\sigma\in K\}$, and $st(\sigma,K)=\sigma\ast lk(\sigma,K)$ will denote the (closed) \emph{star} of $\sigma$ in $K$. The union of two complexes $K, L$ will be denoted by $K+L$.

Following \cite{Gla}, arbitrary subdivisions of $K$ will be denoted by $\alpha K,\ \beta K,\ldots$ Derived subdivisions will be denoted by $\delta K$ and the barycentric subdivision by $K'$, as usual. If $\sigma\in K$ and $a\in\overset{\circ}\sigma$, the interior of $\sigma$, then $(\sigma,a)K$ will denote the elementary subdivision of $K$ by starring $\sigma$ in $a$; i.e. the replacing of $st(\sigma,K)$ by $a\ast \bd{\sigma}\ast lk(\sigma,K)$. A \emph{stellar subdivision} $sK$ of $K$ is a finite sequence of elementary starrings. The operation inverse to an elementary starring is called an \emph{elementary weld} and denoted by $(\sigma,a)^{-1}K$. Two complexes $K$ and $L$ are \emph{stellar equivalent} if they are related by a sequence of starrings, welds and (simplicial) isomorphisms. In this case we write $K\sim L$. It is well known that the combinatorial and the stellar theories are equivalent (see for example \cite{Gla,Lic}),  and therefore $K\sim L$ if and only if they are PL-homeomorphic. A class of complexes will be called \emph{PL-closed} if it is closed under PL-homeomorphisms. 

We recall now the basic definitions and properties of combinatorial manifolds. For a comprehensive exposition of the theory of combinatorial manifolds we refer the reader to \cite{Gla, Lic, Rou}.

$\Delta^n$ will denote the $n$-simplex. A combinatorial $n$-ball is a complex which is PL-homeomorphic to $\Delta^n$. A combinatorial $n$-sphere is a complex PL-homeomorphic to $\partial\Delta^{n+1}$. By convention, $\emptyset=\bd{\Delta}^{0}$ is considered a sphere of dimension $-1$. A combinatorial $n$-manifold is a complex $M$ such that for every $v\in V_M$, $lk(v,M)$ is a combinatorial ($n-1$)-ball or ($n-1$)-sphere. It is easy to verify that  $n$-manifolds are homogeneous complexes of dimension $n$. It is well known that the link of any simplex in a manifold is also a ball or a sphere and that the class of $n$-manifolds is PL-closed. It follows that combinatorial balls and spheres are combinatorial manifolds. 

The boundary  $\bd M$ can be regarded as the set of simplices whose links are combinatorial balls. By a classical result of Newman \cite{New} (see also \cite{Gla,Hud,Lic}), if $S$ is a combinatorial $n$-sphere containing a combinatorial $n$-ball $B$, then the closure $\overline{S-B}$ is a combinatorial $n$-ball.

Some global properties of combinatorial manifolds can be stated in terms of \it pseudo manifolds. \rm An $n$-pseudo manifold is an $n$-homogeneous complex $K$ satisfying the following two properties: (a) for every ($n-1$)-simplex $\sigma$, $lk(\sigma,K)$ is a combinatorial $0$-ball or $0$-sphere (or equivalently, every ($n-1$)-simplex is face of at most two $n$-simplices), and (b) given two $n$-simplices $\sigma,\sigma'$, there exists a sequence of $n$-simplices $\sigma=\sigma_0,\ldots,\sigma_k=\sigma'$ such that $\sigma_i\cap\sigma_{i+1}$ is ($n-1$)-dimensional for all $i=0,\ldots,k-1$. It is well known that any connected combinatorial $n$-manifold, or more generally, any triangulated homological manifold, is an $n$-pseudo manifold.

A simplex $\tau$ of a complex $K$ is said to be \emph{collapsible} in $K$ if it has a free face $\sigma$, i.e. a proper face which is not a face of any other simplex of $K$. Note that, in particular, $\tau$ is a maximal simplex and $\sigma$ is a ridge. In this situation, the operation which transforms $K$ into $K-\{\tau,\sigma\}$ is called an \emph{elementary (simplicial) collapse}, and it is usually denoted by $K\searrow^e K-\{\tau,\sigma\}$. The inverse operation is called an \emph{elementary (simplicial) expansion}. If there is a sequence $K\searrow^e K_1\searrow^e\cdots\searrow^e L$ we say that $K$ \emph{collapses} to $L$ (or equivalently, $L$ \emph{expands} to $K$) and write  $K\searrow L$  or $L\nearrow K$ respectively. A complex $K$ is said to be \emph{collapsible} if it has a subdivision which collapses to a single vertex. A celebrated theorem of J.H.C. Whitehead states that collapsible combinatorial $n$-manifolds are combinatorial $n$-balls \cite[Corollary III.17]{Gla}.

A more general type of collapse is the \emph{geometrical collapse}. If $K=K_0+B^n$, where $B^n$ is a combinatorial $n$-ball and $B^n\cap K_0=B^{n-1}$ is a combinatorial ($n-1$)-ball contained in the boundary of $B^n$, then the move $K\rightarrow K_0$ it called an \emph{elementary geometrical collapse}. A finite sequence of elementary  geometrical collapses (resp. expansions) is a a geometrical collapse (resp. expansion).

If $M$ is an $n$-manifold, an elementary geometrical expansion $M\rightarrow N=M+B^n$ such that $M\cap B^n\subset\bd{M}$ is called an \emph{elementary regular expansion}. By a Theorem of Alexander, an elementary regular expansion is a PL-equivalence (see \cite{Gla, Lic}). A sequence of elementary regular expansions (resp. collapses) is a regular expansion (resp. collapse). Note that the dimension of all the balls being expanded in such a sequence must be $n$.

If $M$ is a combinatorial $n$-manifold with boundary and there is an $n$-simplex $\eta=\sigma\ast \tau\in M$ with $\dim{\sigma},\dim{\tau}\geq 0$ such that $\sigma\in \overset{\circ}{M}$, the interior of $M$, and $\bd{\sigma}\ast \tau\subset\bd{M}$, then the move $M\stackrel{sh}\longrightarrow M_1=\overline{M-\sigma\ast \tau}$ is called an \emph{elementary shelling}. This operation produces again a combinatorial $n$-manifold. The inverse operation is called an \emph{inverse shelling}. Pachner showed in \cite{Pach} that two combinatorial $n$-manifolds with non-empty boundary are PL-homeomorphic if and only if one can  obtain one from the other by a sequence of elementary shellings, inverse shellings and isomorphisms.

\section{NH-manifolds}

A \emph{non-homogeneous manifold}, or \emph{$NH$-manifold} for short, is a simplicial complex which locally looks as in Figure \ref{fig:ejemplos_nh_variedades_locales}. We will define such complexes by induction on the dimension. We need first a definition.

\begin{defi} Let $K$ be a complex. A subcomplex $L\subseteq K$ is said to be \emph{top generated} in $K$ if it is generated by principal simplices of $K$, i.e. every maximal simplex of $L$ is also maximal in $K$.\end{defi}

\begin{figure}[h]
\centering
\includegraphics[width=6.00in,height=1.00in]{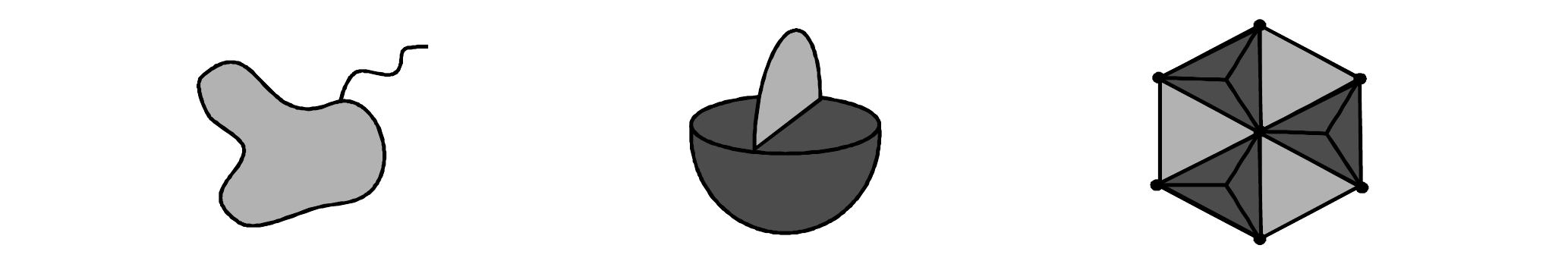}
\caption{Local structure of $NH$-manifolds.}
\label{fig:ejemplos_nh_variedades_locales}
\end{figure}

\begin{defi} An \emph{$NH$-manifold} (resp. \emph{$NH$-ball}, \emph{$NH$-sphere}) of dimension $0$ is a manifold (resp. ball, sphere) of dimension $0$. An $NH$-sphere of dimension $-1$ is, by convention, the empty set. For $n\geq 1$, we define by induction

 \begin{itemize}\item An $NH$-manifold of dimension $n$ is a complex $M$ of dimension $n$ such that $lk(v,M)$ is an $NH$-ball of dimension $0\leq k\leq n-1$ or an $NH$-sphere of dimension $-1\leq k\leq n-1$ for all $v\in V_M$.\item An $NH$-ball of dimension $n$ is a collapsible $NH$-manifold of dimension $n$.\item An $NH$-sphere of dimension $n$ and homotopy dimension $k$ is an $NH$-manifold $S$ of dimension $n$ such that there exist a top generated $NH$-ball $B$ of dimension $n$ and a top generated combinatorial $k$-ball $L$ such that $B + L=S$ and $B\cap L=\bd{L}$. We say that  $S=B+L$ is a \emph{decomposition} of $S$.\end{itemize}\end{defi}

Note that the definition of $NH$-ball is motivated by Whitehead's theorem on regular neighborhoods and the definition of $NH$-sphere by that of Newman's (see \cite{Gla} and \cite{Lic}).

\begin{obs} An $NH$-ball of dimension $1$ is the same as a $1$-ball.  An $NH$-sphere of dimension $1$ is either a $1$-sphere (if the homotopy dimension is $1$) or the disjoint union of a point and a combinatorial $1$-ball (if the homotopy dimension is $0$).  In general, an $NH$-sphere of homotopy dimension $0$ consists of a disjoint union of a point and an $NH$-ball. These are the only $NH$-spheres which are not connected.\end{obs}

\begin{ejs}\label{examples of NH-manifolds} Figure \ref{fig:ejemplos_nh_variedades} shows some examples of $NH$-manifolds.

\begin{figure}[h]
\centering
\includegraphics[width=6.00in,height=1.00in]{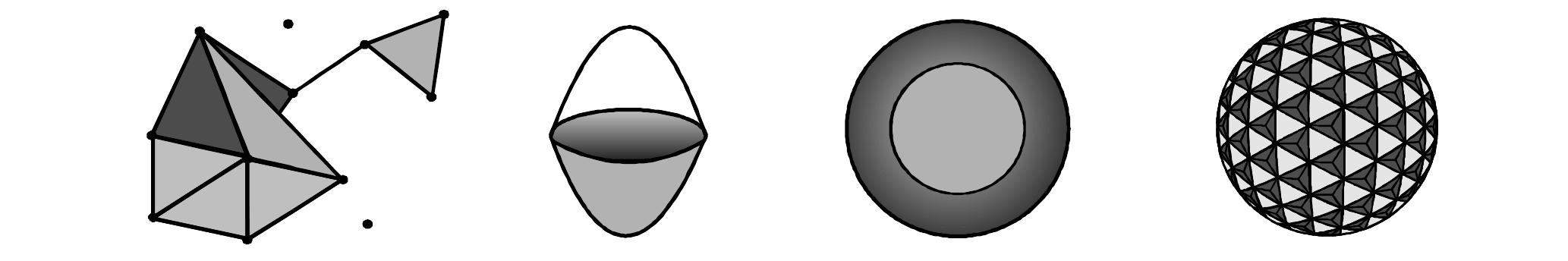}
\caption{$NH$-manifolds.}
\label{fig:ejemplos_nh_variedades}
\end{figure}

 \end{ejs}

\begin{obs} Note that the decomposition of an $NH$-sphere need not be unique. However the homotopy dimension of the $NH$-sphere is well defined since the geometric realization of an $NH$-sphere of homotopy dimension $k$ is a homotopy $k$-sphere. \end{obs}

We show now that the notion of $NH$-manifold is in fact an extension of the concept of combinatorial manifold to the non-homogeneous context.

\begin{teo}\label{thm homogeneous NH-manifolds} A complex $K$ is a homogeneous $NH$-manifold (resp. $NH$-ball, $NH$-sphere) of dimension $n$ if and only if it is a combinatorial $n$-manifold (resp. $n$-ball, $n$-sphere).\end{teo}

\begin{proof} Let $n\geq 1$. It is easy to see that the result holds for $NH$-manifolds of dimension $n$ provided that it holds for $NH$-balls and $NH$-spheres of dimension less than $n$. Then it suffices to prove that the result holds for $NH$-balls and $NH$-spheres of dimension $n$ if it holds for $NH$-manifolds of dimension $n$.

For $NH$-balls the result is clear by the theorem of Whitehead \cite[Corollary III.17]{Gla}. Suppose now that $S=B+L$ is a homogeneous $NH$-sphere of dimension $n$. It follows that $B$ and $L$ are combinatorial $n$-balls. Take $\sigma\in\bd{L}$ a maximal simplex. Since $lk(\sigma,S)=\{v\}+ lk(\sigma,B)$ for some vertex $v\in L$ and $S$ is an $n$-pseudo manifold, then $lk(\sigma,B)$ is also a single vertex. It follows that $\sigma\in\bd{B}$. Since both $\bd{L}$ and $\bd{B}$ are combinatorial ($n-1$)-spheres, this implies that $\bd{L}=\bd{B}$. This proves that $S$ is a combinatorial $n$-sphere. Conversely, any $n$-simplex of a combinatorial $n$-sphere can play the role of $L$ in its decomposition as an $NH$-sphere. The result then follows from Newman's Theorem.\end{proof}

Following the same reasoning of \cite[Theorem II.2]{Gla} for combinatorial manifolds, one can show that the links of all simplices in an $NH$-manifold behave nicely. Concretely:

\begin{prop}\label{prop regularity of NH-manifolds} Let $M$ be an $NH$-manifold of dimension $n$ and let $\sigma\in M$ be a $k$-simplex. Then $lk(\sigma,M)$ is an $NH$-ball or an $NH$-sphere of dimension less than $n-k$.\end{prop}

The property stated in the preceding proposition is often called \emph{regularity}. 

In order to show that the class of $NH$-manifolds is PL-closed, we will need the following lemma, which is somehow an analogue of \cite[Proposition II.1]{Gla}. This result will be generalized in Corollary \ref{join of balls and spheres with q-balls and q-spheres} and in Theorem \ref{thm join of q-balls and q-spheres}.

\begin{lema}\label{lemma simplex join q-ball q-sphere} Let $K$ be an $NH$-ball or an $NH$-sphere and let $\sigma$ be a simplex disjoint from $K$. Then, \begin{enumerate}\item $\sigma\ast K$ is an $NH$-ball.\item $\bd{\sigma}\ast K$ is an $NH$-ball (if $K$ is an $NH$-ball) or an $NH$-sphere (if $K$ is an $NH$-sphere).\end{enumerate}\end{lema}

\begin{proof} For the first part of the lemma, we proceed by double induction.  Suppose first that $\dim{\sigma}=0$, i.e. $\sigma$ is a vertex $v$, and that the result holds for $NH$-balls and $NH$-spheres $K$ of dimension less than $n$. Note that $v\ast K\searrow 0$, so we only have to verify that $v\ast K$ is an $NH$-manifold.
Take $w\in V_K$. Since $lk(w,v\ast K)=v\ast lk(w,K)$, by induction applied to $lk(w,K)$, it follows that $lk(w,v\ast K)$ is an $NH$-ball. On the other hand, $lk(v,v\ast K)=K$, which is an $NH$-ball or an $NH$-sphere by hypothesis. This shows that $v\ast K$ is an $NH$-manifold and proves the case $\dim{\sigma}=0$. Suppose now that $\dim{\sigma}=k\geq 1$. Write $\sigma=\tau\ast v$ for some $v\in\sigma$. Since $\sigma\ast K=\tau\ast(v\ast K)$, the results follows by induction applied to $v$ and $\tau$. 

For the second part of the lemma, suppose that $\dim{\sigma}=k\geq 1$ and let $K$ be an $NH$-ball or an $NH$-sphere of dimension $n$. It is easy to see that the result is valid if $n=0$. Suppose now that $n\geq 1$ and that the result holds for $t<n$. For any vertex $v\in\bd{\sigma}\ast K$, we have
\begin{equation*} lk(v,\bd{\sigma}\ast K)=\left\{\begin{array}{ll}
  \bd{\sigma}\ast lk(v,K) & v\notin\bd{\sigma} \\
  lk(v,\bd{\sigma})\ast K & v\in\bd{\sigma} \\
\end{array}
\right.\end{equation*} 
In the first case,  by induction on $n$, it follows that $lk(v,\bd{\sigma}\ast K)$ is an $NH$-ball or sphere. In the second case, we use induction on $k$ (note that $lk(v,\bd{\sigma})=\bd{lk(v,\sigma)}$). This proves that $\bd{\sigma}\ast K$ is an $NH$-manifold. Now, if $K$ is an $NH$-ball then $\bd{\sigma}\ast K\searrow 0$ and $\bd{\sigma}\ast K$ is again an $NH$-ball. If $K$ is an $NH$-sphere write $K=B+L$ with $B$ an $NH$-ball, $L$ a combinatorial ball and $B\cap L=\partial L$. Since $\bd(\bd{\sigma}\ast L)=\bd{\sigma}\ast\bd{L}=\bd{\sigma}\ast B\cap\bd{\sigma}\ast L$, then $\bd{\sigma}\ast K=\bd{\sigma}\ast B+\bd{\sigma}\ast L$ is an $NH$-sphere by the previous case. This concludes the proof.\end{proof}

In particular, from Lemma \ref{lemma simplex join q-ball q-sphere} we deduce that $M$ is an $NH$-manifold if and only if $st(v,M)$ is an $NH$-ball for all $v\in V_M$.


\begin{teo}\label{thm pl-invariancy of NH-manifolds} The classes of $NH$-manifolds, $NH$-balls and $NH$-spheres are PL-closed.\end{teo}

\begin{proof} It suffices to prove that $K$ is an $NH$-manifold (resp. $NH$-ball, $NH$-sphere) if and only if any starring $(\tau,a)K$ is an $NH$-manifold (resp. $NH$-ball, $NH$-sphere). We suppose first that the result is valid for $NH$-manifolds of dimension $n$ and prove that it is valid for $NH$-balls and $NH$-spheres of the same dimension. If $(\tau,a)K$ is an $NH$-ball of dimension $n$ then $K$ is also an $NH$-ball since it is an $NH$-manifold with $\alpha((\tau,a)K)\searrow 0$ for some subdivision $\alpha$. On the other hand, if $K$ is an $NH$-manifold of dimension $n$ with  $\alpha K\searrow 0$, by \cite[Theorem I.2]{Gla} we can find a stellar subdivision $\delta$ and an arbitrary subdivision $\beta$ such that $\beta((\tau,a)K)=\delta(\alpha K)$. Since stellar subdivisions preserve collapses, $(\tau,a)K$ is collapsible and hence an $NH$-ball. Now, if $K$ is an $NH$-sphere of dimension $n$ with decomposition $B+L$ then the result holds by the previous case and the following identities. \begin{equation*}(\tau,a)K=\left\{\begin{array}{ll}
  (\tau,a)B+L \mbox{, with $(\tau,a)B\cap L=\bd{L}$ }  & a\in B-L \\
  & \\
  B+(\tau,a)L \mbox{, with $B\cap(\tau,a)L=\bd{L}$ }  & a\in L-B \\
  & \\
  (\tau,a)B+(\tau,a)L \mbox{, with $(\tau,a)B\cap(\tau,a)L=(\tau,a)\partial L$ }  & a\in B\cap L=\bd{L}\\
\end{array}\right.\end{equation*}
Note that $(\tau,a)\bd{L}=\bd(\tau,a)L$. The converse follows by replacing $(\tau,a)$ with $(\tau,a)^{-1}$.

We assume now that the result is valid for $NH$-balls and $NH$-spheres of dimension $n$ and prove that it is valid for $NH$-manifolds of dimension $n+1$. Suppose $K$ is an $NH$-manifold of dimension $n+1$ and take $v\in(\tau,a)K$. If $v\neq a$ then $lk(v,(\tau,a)K)$ is PL-homeomorphic to an elementary starring of $lk(v,K)$ . The inductive hypothesis on $lk(v,K)$ shows  that $lk(v,(\tau,a)K)$ is also an $NH$-ball or $NH$-sphere. On the other hand, $lk(a,(\tau,a)K)=\bd{\tau}\ast lk(\tau,K)$, which is an $NH$-ball or an $NH$-sphere by Lemma \ref{lemma simplex join q-ball q-sphere}. Once again, the converse follows by replacing  $(\tau,a)$ with $(\tau,a)^{-1}$.\end{proof}

\begin{coro}\label{join of balls and spheres with q-balls and q-spheres} Let $B$ be a combinatorial $n$-ball, $S$ a combinatorial $n$-sphere and $K$ an $NH$-ball or $NH$-sphere. Then, \begin{enumerate}\item $B\ast K$ is an $NH$-ball.\item $S\ast K$ is an $NH$-ball (if $K$ is an $NH$-ball) or an $NH$-sphere (if $K$ is an $NH$-sphere).\end{enumerate} \end{coro}

\begin{proof} Follows from Lemma \ref{lemma simplex join q-ball q-sphere} and  Theorem \ref{thm pl-invariancy of NH-manifolds}. \end{proof}

\begin{prop}\label{prop of pseudofrontier} Let $K$ be an $n$-dimensional complex and let $B$ be a combinatorial $r$-ball. Suppose $K+B$ is an $NH$-manifold such that
\begin{enumerate}
\item $K\cap B\subset\bd{B}$ is homogeneous of dimension $r-1$ and
\item $lk(\sigma,K)$ is collapsible for all $\sigma \in K\cap B$
\end{enumerate}
Then, $K$ is an $NH$-manifold.\end{prop}

\begin{proof} We show first that $K,B\subset K+B$ are top generated. Clearly, $B$ is top generated since it intersects $K$ in dimension $r-1$. On the other hand, a principal simplex in $K$ which is not principal in $K+B$ must lie in $K\cap B$. Then, by hypothesis, it has a collapsible link in $K$. But this contradicts the fact that it is principal in $K$. Therefore $K,B\subset K+B$ are top generated and, in particular, $r\leq n$. 

We prove the result by induction on $r$. For $r=0$ the result is trivial. Let $r\geq 1$ and $v\in K$. If $v\notin B$ then $lk(v,K)=lk(v,K+B)$, which is an $NH$-ball or $NH$-sphere by hypothesis. Suppose now that $v\in K\cap B$. If $r=1$, then $lk(v,K+B)=lk(v,K)+\ast$. It follows that  $lk(v,K)$ is an $NH$-ball. Suppose $r\geq 2$ (and hence $n\geq 2$). We will see that the pair $lk(v,K),lk(v,B)$ also satisfies the conditions of the theorem. Note that $lk(v,K)+lk(v,B)=lk(v,K+B)$ is an $NH$-manifold by hypothesis and 
$lk(v,K)\cap lk(v,B)=lk(v,K\cap B)$ is homogeneous of dimension $r-2$. On the other hand, if $\eta\in lk(v,K)\cap lk(v,B)$ then, $v\ast\eta\in K\cap B$, so $lk(\eta,lk(v,K))=lk(v\ast\eta,K)$ is collapsible. By induction, it follows that  $lk(v,K)$ is an $NH$-manifold, and, since it is also collapsible, it is an $NH$-ball. This shows that $K$ is an $NH$-manifold.\end{proof}

\begin{lema}\label{lemma previous to joins} Suppose $S_1=G_1+L_1$ and $S_2=G_2+L_2$ are two disjoint $NH$-spheres. Then, $G_1\ast S_2+L_1\ast G_2$ is collapsible.\end{lema}

\begin{proof} Since $G_1$ and $G_2$ are collapsible, there exist subdivisions $\epsilon_1,\epsilon_2$ such that $\epsilon_1 G_1\searrow 0$ and $\epsilon_2 G_2\searrow 0$. We can extend these subdivisions to $S_1$ and $S_2$ and then suppose without loss of generality that $G_1\searrow 0$ and $G_2\searrow 0$. Note that $$G_1\ast S_2\cap L_1\ast G_2=\bd{L}_1\ast G_2.$$
We will show that some subdivision of $L_1\ast G_2$ collapses to (the induced subdivision of) $\bd{L}_1\ast G_2$. Let $\alpha$ be an arbitrary subdivision of $L_1$ and $\delta$ a derived subdivision of $\Delta^r$ such that $\alpha L_1=\delta\Delta^r$. Then, $\alpha(L_1\ast G_2)=\delta(\Delta^r\ast G_2)$. Since $G_2\searrow 0$, then $\Delta^r\ast G_2\searrow\bd{\Delta}^r\ast G_2$ (\cite[Corollary III.4]{Gla}). Therefore
$$\alpha(L_1\ast G_2) =  \delta(\Delta^r\ast G_2) \searrow  \delta(\bd{\Delta}^r\ast G_2) =  \alpha(\bd{L}_1\ast G_2).$$
We extend $\alpha$ to $(G_1\ast S_2+L_1\ast G_2)$ and then
$$\alpha(G_1\ast S_2+L_1\ast G_2)=\alpha(G_1\ast S_2)+\alpha(L_1\ast G_2)\searrow\alpha(G_1\ast S_2)+\alpha(\bd{L}_1\ast G_2)=\alpha(G_1\ast S_2).$$
By \cite[Theorem III.6]{Gla} there is a stellar subdivision $s$ such that $s\alpha G_1\searrow 0$ and therefore
$$s\alpha(G_1\ast S_2+L_1\ast G_2)\searrow s\alpha(G_1\ast S_2)=s\alpha G_1\ast s\alpha S_2\searrow 0.$$
\end{proof}

\begin{teo}\label{thm join of q-balls and q-spheres} Let $B_1,B_2$ be $NH$-balls and $S_1,S_2$ be $NH$-spheres. Then, \begin{enumerate}\item $B_1\ast B_2$ and $B_1\ast S_2$ are $NH$-balls.\item $S_1\ast S_2$ is an $NH$-sphere.\end{enumerate}\end{teo}

\begin{proof} Let $K_1$ represent $B_1$ or $S_1$ and let $K_2$ represent $B_2$ or $S_2$. We must show that $K_1\ast K_2$ is an $NH$-ball or an $NH$-sphere. We proceed by induction on $s=\dim{K_1}+\dim{K_2}$. If $s=0,1$ the result follows from Lemma \ref{lemma simplex join q-ball q-sphere}. Let $s\geq 2$. We show first that $K_1\ast K_2$ is an $NH$-manifold. Let $v\in K_1\ast K_2$ be a vertex. Then,
$$lk(v,K_1\ast K_2)=\left\{
\begin{array}{ll}
lk(v,K_1)\ast K_2 & v\in K_1 \\
K_1\ast lk(v,K_2) & v\in K_2
\end{array}
\right.$$
Since $\dim{lk(v,K_1)}+\dim{K_2}=\dim{K_1}+\dim{lk(v,K_2)}=s-1$, then by induction, $lk(v,K_1\ast K_2)$ is an $NH$-ball or an $NH$-sphere. It follows that $K_1\ast K_2$ is an $NH$-manifold. Now, if $K_1=B_1$ or $K_2=B_2$, then $K_1\ast K_2\searrow 0$ and  $K_1\ast K_2$ is an $NH$-ball.

We prove now that $S_1\ast S_2$ is an $NH$-sphere. Decompose $S_1=G_1+L_1$ and $S_2=G_2+L_2$. Note that $S_1\ast S_2 = (G_1\ast S_2+L_1\ast G_2)+L_1\ast L_2$ and that $(G_1\ast S_2+L_1\ast G_2)\cap (L_1\ast L_2) = \bd(L_1\ast L_2)$, then it suffices to show that $(G_1\ast S_2+L_1\ast G_2)$ is an $NH$-ball. By Lemma \ref{lemma previous to joins} it is collapsible, so we only need to check that $(G_1\ast S_2+L_1\ast G_2)$ is an $NH$-manifold. In order to prove this, we apply Proposition \ref{prop of pseudofrontier} to the complex $G_1\ast S_2+L_1\ast G_2$ and the combinatorial ball $L_1\ast L_2$. The only non-trivial fact is that $lk(\sigma,G_1\ast S_2+L_1\ast G_2)$ is collapsible for $\sigma\in\bd(L_1\ast L_2)$. To see this, take $\eta\in \bd(L_1\ast L_2)=\bd{L}_1\ast L_2+L_1\ast\bd{L}_2$ and write $\eta=l_1\ast l_2$ with $l_1\in L_1$, $l_2\in L_2$. Then,
$$lk(\eta,G_1\ast S_2+L_1\ast G_2)= lk(l_1,G_1)\ast lk(l_2,S_2)+lk(l_1,L_1)\ast lk(l_2,G_2).$$
Now, if $l_1\in L_1-\bd{L}_1$ then $lk(l_1\ast l_2,G_1\ast S_2)=\emptyset$ and $lk(\eta,G_1\ast S_2+L_1\ast G_2)=lk(l_1,L_1)\ast lk(l_2,G_2)\searrow 0$. By a similar argument, the same holds if $l_2\in L_2-\bd{L}_2$. If $l_1\in\bd{L}_1$ and $l_2\in\bd{L}_2$ then $lk(l_1,S_1)=lk(l_1,G_1)+lk(l_1,L_1)$ and $lk(l_2,S_2)=lk(l_2,G_2)+lk(l_2,L_2)$ are $NH$-spheres (by Lemma \ref{lemma border of q-spheres}). By Lemma \ref{lemma previous to joins}, it follows that $lk(\eta,G_1\ast S_2+L_1\ast G_2)$ is also collapsible. By Proposition \ref{prop of pseudofrontier}, we conclude that $G_1\ast S_2+L_1\ast G_2$ is an $NH$-manifold.\end{proof}

The following result will be used in the next section. First we need a definition.

\begin{defi}
Two principal simplices $\sigma,\tau\in M$ are said to be \textit{adjacent} if the intersection $\tau\cap\sigma$ is an immediate face of $\sigma$ or $\tau$.
\end{defi}

\begin{lema}\label{prop connected q-manifold is q-pseudo} Let $M$ be a connected $NH$-manifold. Then
\begin{enumerate}\item For each ridge $\sigma\in M$, $lk(\sigma,M)$ is either a point or an $NH$-sphere of homotopy dimension $0$.\item Given any two principal simplices $\sigma,\tau\in M$, there exists a sequence $\sigma=E_1,\ldots,E_s=\tau$ of principal simplices of $M$ such that $E_i$ is adjacent to $E_{i+1}$ for every $1\leq i\leq s-1$.\end{enumerate}\end{lema}

By analogy with the homogeneous case, a complex $K$ satisfying properties (1) and (2) of this lemma will be called an \emph{$NH$-pseudo manifold}. For more details on (homogeneous) pseudo manifolds we refer the reader to \cite{Mun} (see also \cite{Spa}). The proof of Lemma \ref{prop connected q-manifold is q-pseudo} will follow from the next result.

\begin{lema}\label{lemma star q-pseudo is q-pseudo} If $K$ is a connected complex such that $st(v,K)$ is an $NH$-pseudo manifold for all $v\in V_K$ then $K$ is an $NH$-pseudo manifold. \end{lema}

\begin{proof} We will show that $K$ satisfies properties (1) and (2) of Lemma \ref{prop connected q-manifold is q-pseudo}.
Let $\sigma\in K$ be a ridge and let $v\in\sigma$ be any vertex. Then $\sigma$ is also a ridge in $st(v,K)$ and $lk(\sigma,K)=lk(\sigma,st(v,K))$. Therefore $K$ satisfies property (1).

Let $\nu,\tau\in K$ be maximal simplices and let $v\in\nu,\ w\in\tau$. Take an edge path from $v$ to $w$. We will prove that $K$ satisfies property (2) by induction on the length $r$ of the edge path. If $r=0$, then $v=w$. In this case, $\nu,\tau\in st(v,K)$ and the results follows by hypothesis. Suppose now that $\psi_1,\ldots,\psi_r$ is an edge path from $v$ to $w$ of length $r\geq 1$. Take maximal simplices $E_i$ such that $\psi_i\leq E_i$. Note that $E_1\cap E_2$ contains the vertex $\psi_1\cap \psi_2$. By hypothesis, $st(\psi_1\cap\psi_2,K)$ satisfies property (2) and therefore we can join $E_1$ with $E_2$ by a sequence of adjacent maximal simplices. Now the result follows by induction.
\end{proof}

\begin{proof}[Proof of Lemma \ref{prop connected q-manifold is q-pseudo}] We proceed by induction on the dimension $n$ of $M$. By Lemma \ref{lemma star q-pseudo is q-pseudo}, it suffices to prove that $st(v,M)$ is an $NH$-pseudo manifold for every vertex $v$. The case $n=0$ is trivial. Suppose that $n\geq 1$ and that the result is valid for $k\leq n-1$. Now, if $lk(v,M)$ is an $NH$-ball or a connected $NH$-sphere then, by induction, it is an $NH$-pseudo manifold. It follows that $st(v,M)$ is also an $NH$-pseudo manifold since it is a cone of an $NH$-pseudo manifold. In the other case, $lk(v,M)$ is an $NH$-sphere of homotopy dimension $0$ of the form $B+\ast$, for some $NH$-ball $B$. Since $vB$ is an $NH$-pseudo manifold, it follows that  $st(v,M)$ is also an $NH$-pseudo manifold.\end{proof}

\section{Boundary, pseudo boundary and the Anomaly Complex}

The concept of boundary is not defined in the non-homogeneous setting and, in fact, it is not clear what a boundary of a general complex could be. However, the characterization of the boundary of combinatorial manifolds allows us to extend this notion to the class of $NH$-manifolds.

\begin{defi} Let $M$ be an $NH$-manifold. The \emph{pseudo boundary} of $M$ is the set of simplices $\pbd M$ whose links are $NH$-balls. The \emph{boundary} of $M$ is the subcomplex $\bd M$ spanned by $\pbd M$. In other words,  $\bd{M}$ is the closure $\overline{\pbd M}$.\end{defi}

\begin{figure}[h]
\centering
\includegraphics[width=6.00in,height=2.00in]{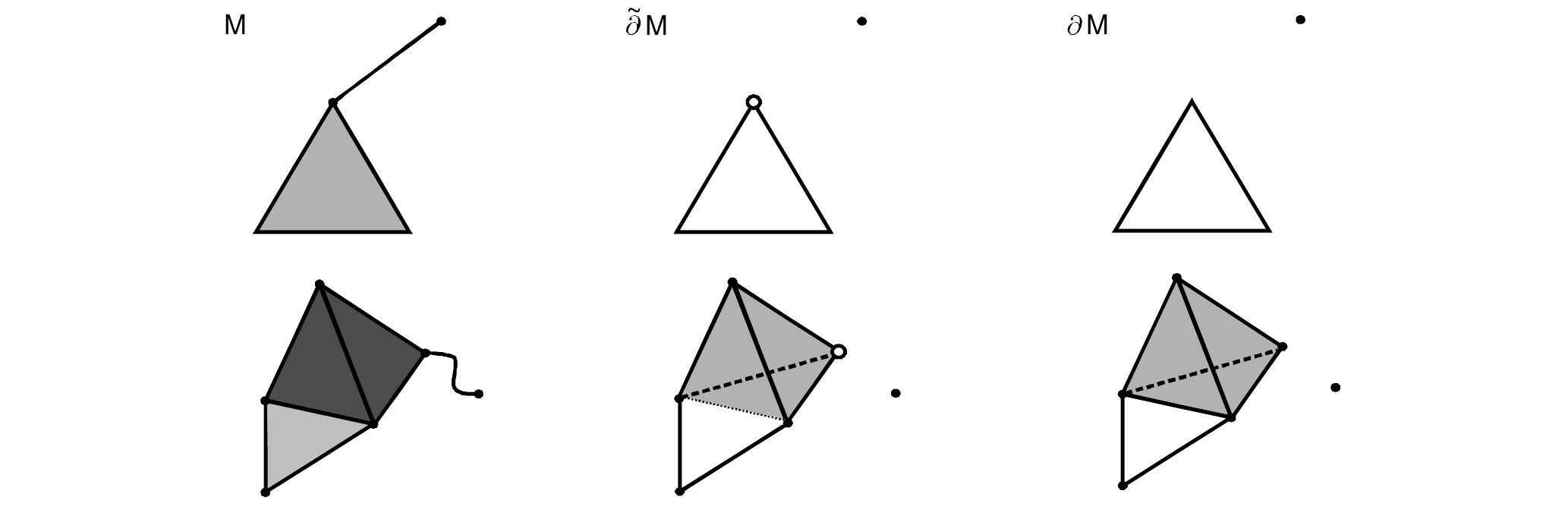}
\caption{Boundary and pseudo boundary.}
\label{fig:ejemplos_borde_y_pseudo}
\end{figure}


It is clear that $\pbd M=\bd{M}$ for any combinatorial manifold $M$. We will see that, in fact, this is the only case where this happens. The result will follow from the next lemma.

\begin{lema}\label{lemma simplex in two principal is in boundary} Let $M$ be an $NH$-manifold and let $\sigma\in M$. If $\sigma$ is a face of two principal simplices of different dimensions then $\sigma\in\partial M$.\end{lema}

\begin{proof} Let $\tau_1=\sigma\ast\eta_1$ and $\tau_2=\sigma\ast\eta_2$ be principal simplices such that $\dim{\tau_1}\neq\dim{\tau_2}$. By Lemma \ref{prop connected q-manifold is q-pseudo} we may assume that $\tau_1$ and $\tau_2$ are adjacent. Let $\rho=\tau_1\cap\tau_2$ and suppose $\rho\prec\tau_1$. Then, $lk(\rho,M)$ is an $NH$-sphere of homotopy dimension $0$ with decomposition $lk(\rho,M-\tau_1)+\ast$. Since $\dim{lk(\rho,M-\tau_1)}\geq 1$ then $\pbd{lk(\rho,M-\tau_1)}=\pbd{lk(\rho,M)}$ is non-empty. For any simplex $\nu$ in $\pbd{lk(\rho,M)}$,  $\nu\ast\rho\in\tilde{\partial} M$. Thus $\sigma\in\partial M$.\end{proof}

\begin{prop}\label{prop boundary homogeneous} If $M$ is a connected $NH$-manifold such that $\pbd M=\bd{M}$ then $M$ is a combinatorial manifold. In particular, $NH$-manifolds without boundary (or pseudo boundary) are combinatorial manifolds.\end{prop}

\begin{proof} If $M$ is non-homogeneous, by Lemma \ref{prop connected q-manifold is q-pseudo}  there exist two adjacent principal simplices $\tau_1,\tau_2$ of different dimensions. By Lemma \ref{lemma simplex in two principal is in boundary}, $\rho=\tau_1\cap\tau_2\in\partial M-\tilde{\partial}M$.\end{proof}

The following result will be used in the next sections. It is the non-homogeneous version of the well-known fact that any $n$-homogeneous subcomplex of an $n$-combinatorial manifold with non-empty boundary has also a non-empty boundary (see \cite{Gla}).

\begin{lema}\label{lemma q-submanifold contained in q-manifold} Let $M$ be a connected $NH$-manifold with non-empty boundary and let $L\subseteq M$ be a top generated $NH$-submanifold. Then, $\partial L\neq\emptyset$.\end{lema}

\begin{proof} We may assume $L\neq M$. We proceed by induction on $n=\dim{M}$. The $1$-dimensional case is clear. Let $n\geq 2$. Take adjacent principal simplices $\sigma\in L$ and $\tau\in M-L$ and let $\rho=\sigma\cap\tau$. If $\dim{\sigma}=\dim{\tau}$ then $lk(\rho,M)=S^0$ and therefore, $\rho\in\bd{L}$. If $\dim{\sigma}\neq\dim{\tau}$ then $lk(\rho,M)=B+\ast$ is a non-homogeneous $NH$-sphere of homotopy dimension $0$. We analyze both cases: $\rho\prec\sigma$ and $\rho\prec\tau$. If $\rho\prec\sigma$ then $lk(\rho,L)$ is either a $0$-ball, which implies $\rho\in\pbd{L}$, or a non-homogeneous $NH$-sphere of homotopy dimension $0$. In this case, by Proposition \ref{prop boundary homogeneous} $\pbd{lk(\rho,L)}\neq\emptyset$. If $\rho\prec\tau$ then $\pbd{lk(\rho,L)}\neq\emptyset$ by induction applied to $lk(\rho,L)\subset B$. In any case, if $\eta\in\pbd{lk(\rho,L)}$ then $\eta\ast\rho\in\pbd{L}$. \end{proof}

\begin{coro}\label{coro boundaryless in nhmanifold} If $M$ is a connected $NH$-manifold of dimension $n\geq 1$ containing a top generated combinatorial manifold $L$ without boundary then $M=L$.\end{coro}

Note that if $S=B+\ast$ is a non-homogeneous $NH$-sphere of homotopy dimension $0$ and $M$ is a non-trivial top generated combinatorial $n$-manifold contained in $S$, then $M\subseteq B$. This implies that $\bd{M}\neq\emptyset$ by Corollary \ref{coro boundaryless in nhmanifold}. We state this fact in the following

\begin{coro}\label{coro nhbouquet cannot sphere} A non-homogeneous $NH$-sphere $S=B+\ast$ of homotopy dimension $0$ cannot contain a non-trivial top generated combinatorial manifold without boundary.\end{coro}

 In contrast to the classical situation, the boundary of an $NH$-manifold is not in general an $NH$-manifold (see Figure \ref{fig:ejemplos_borde_y_pseudo}). However, similarly as in the homogeneous setting, if $M$ is an $NH$-manifold and $\eta\in M$ is any simplex, then $lk(\eta,\bd{M})=\bd{lk(\eta,M)}$. Moreover, it is well-known that the boundary of a combinatorial manifold has no boundary. The following result generalizes this fact to the non-homogeneous setting.

\begin{prop}\label{prop borde del borde} The boundary of an $NH$-manifold $M$ has no collapsible simplices.\end{prop}

\begin{proof} Let $\sigma$ be a ridge in $\bd{M}$. Since $lk(\sigma,\partial M)=\partial lk(\sigma,M)$,  it suffices to show that the boundary of any $NH$-ball or  $NH$-sphere cannot be a singleton. This is clear for classical balls and spheres and, by Proposition \ref{prop boundary homogeneous}, the same is true for $NH$-balls and $NH$-spheres.\end{proof}

A simplex $\sigma\in M$ will be called \emph{internal} if $lk(\sigma,M)$ is an $NH$-sphere, i.e. if $\sigma\notin\pbd M$. We denote by $\overset{\circ}{M}$ the relative interior of $M$, which is the set of its internal simplices.

\begin{lema}\label{lemma border of q-spheres} Let $S$ be an $NH$-sphere with decomposition $B+L$. Then, every $\sigma\in L$ is internal in $S$. In particular, $\pbd S=\pbd B-L$.\end{lema}

\begin{proof} This is a particular case of Lemma \ref{lemma border of q-bouquets}.\end{proof}

\begin{defi} Let $M$ be an $NH$-manifold. The \emph{anomaly complex} of $M$ is the subcomplex
$$A(M)=\{\sigma\in M\mbox{ $:$ }lk(\sigma,M)\mbox{ is not homogeneous}\}.$$\end{defi}

The fact that $A(M)$ is a simplicial complex follows  from the equation $lk(\sigma\ast\eta,M)=lk(\sigma,lk(\eta,M))$. Figure \ref{fig:ejemplos_anomaly} shows examples of anomaly complexes.

\begin{figure}[h]
\centering
\includegraphics[width=6.00in,height=1.00in]{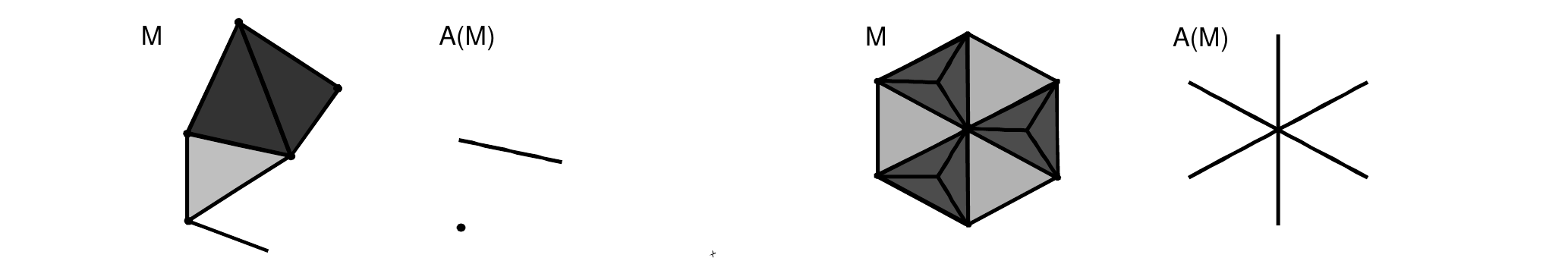}
\caption{Anomaly complex.}
\label{fig:ejemplos_anomaly}
\end{figure}

\begin{prop} For any $NH$-manifold $M$, $\bd{M}=\pbd M + A(M)$.\end{prop}

\begin{proof} If $\sigma\in A(M)$ then $\sigma$ is face of two principal simplices of $M$ of different dimensions. Therefore $\sigma\in\bd{M}$ by Lemma \ref{lemma simplex in two principal is in boundary}. For the other inclusion, let $\sigma\in\bd{M}-\pbd M$. Then $lk(\sigma,M)$ is an $NH$-sphere and $\sigma<\tau$ with $\tau\in\pbd M$. Write $\tau=\sigma\ast\eta$, thus $lk(\tau,M)=lk(\eta,lk(\sigma,M))$. If $\sigma\notin A(M)$ then $lk(\sigma,M)$ is a combinatorial sphere and so is $lk(\tau,M)$, contradicting the fact that $\tau\in\pbd M$.\end{proof}

\section{NH-bouquets and shellability}

Recall that, similarly as in the homogeneous setting,  an $NH$-sphere is obtained by ``gluing" a combinatorial ball to an $NH$-ball along its entire boundary. In the homogeneous case one can no longer glue another  ball to a sphere for it would produce a complex which is not a manifold (not even a pseudo manifold). The existence of boundary in non-homogeneous $NH$-spheres allows us to glue balls and obtain again an $NH$-manifold. This is the idea behind the notion of $NH$-bouquet. This concept arises naturally when studying shellability of non-homogeneous manifolds.

\begin{defi} We define an \emph{$NH$-bouquet} $G$ of dimension $n$ and \emph{index} $k$ by induction on $k$.\begin{itemize}\item If $k=0$ then $G$ is an $NH$-ball of dimension $n$.\item If $k\geq 1$ then $G$ is an $NH$-manifold of dimension $n$ such that there exist a top generated $NH$-bouquet $S$ of dimension $n$ and index $k-1$ and a top generated combinatorial ball $L$, such that $G=S+ L$ and $S\cap L=\bd{L}$.\end{itemize}\end{defi}

 We will show below that the index $k$ is well defined since an $NH$-bouquet of index $k$ is homotopy equivalent to a bouquet of $k$ spheres (of different dimensions). In fact, the index is the number of balls that are glued to an $NH$-ball.  A \emph{decomposition}  $G=B+L_1+\cdots+L_k$ of an $NH$-bouquet $G$ consists of top generated subcomplexes of $G$ such that $B$ is an $NH$-ball, $L_i$ is a combinatorial ball for each $i=1,\ldots,k$ and $(B+\cdots+L_i)\cap L_{i+1}=\bd L_{i+1}$. Of course, a decomposition is not unique.

\begin{ejs} Figure \ref{fig:ejemplos_nh_bouquets} shows some examples of $NH$-bouquets of low dimensions.\end{ejs}

\begin{figure}[h]
\centering
\includegraphics[width=6.00in,height=1.00in]{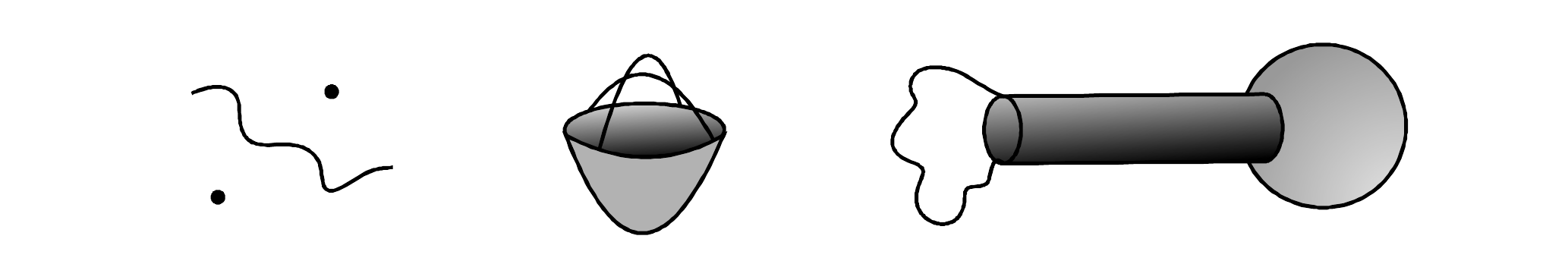}
\caption{$NH$-bouquets.}
\label{fig:ejemplos_nh_bouquets}
\end{figure}

\begin{obs} Clearly an $NH$-bouquet of index $1$ is an $NH$-sphere. Note also that for every $n\geq 0$ and every $k\geq 0$ there exists an $NH$-bouquet $G$ of dimension $n$ and index $k$.
\end{obs}

Similarly  as in Theorem \ref{thm pl-invariancy of NH-manifolds}, it can be proved that the class of $NH$-bouquets is PL-closed. 

\begin{lema}\label{lemma central of q-bouquets} If $G=B+L_1+\cdots+L_k$ is a decomposition of an $NH$-bouquet of index $k\geq 2$, then $L_i\cap L_j=\bd{L}_i\cap\bd{L}_j$ for all $1\leq j<i\leq k$.\end{lema}

\begin{proof}
$L_i\cap L_j\subseteq \bd L_i$ by definition. Suppose that $L_i\cap L_j\nsubseteq \bd L_j.$ Then there exists a simplex $\sigma\in L_i\cap L_j$ such that $lk(\sigma,L_j)$ is a sphere. By Corollaries \ref{coro boundaryless in nhmanifold} and \ref{coro nhbouquet cannot sphere}, $lk(\sigma,L_j)=lk(\sigma,G)$. In particular $lk(\sigma,L_i)\subseteq lk(\sigma,L_j)$, but if $\nu\in lk(\sigma, L_i)$ is maximal, then $\sigma\ast\nu$ is a maximal simplex in $G$ and it is contained in $L_i\cap L_j\subseteq \bd L_i$ which is a contradiction.
\end{proof}

\begin{prop}\label{corollary homotopy type of q-bouquets} If $G=B+L_1+\cdots +L_k$ is a decomposition of an $NH$-bouquet, then $\bd{L}_i\subseteq B$ for every $i=1,\ldots,k$. In particular, an $NH$-bouquet of index $k$ is homotopy equivalent to a bouquet of spheres of dimensions $\dim{L_i}$, for $1\leq i\leq k$.\end{prop}

\begin{proof} $\bd{L}_1\subseteq B$ by definition. For $i\geq 2$ the result follows immediately by induction and Lemma \ref{lemma central of q-bouquets}.

For the second statement, note that, since $\bd L_i\subseteq B$ for every $i$, $G$ is homotopy equivalent to a CW-complex obtained by attaching cells of dimensions $\dim{L_i}$ to a point.
\end{proof}

\begin{obs} It is not hard to see that a homogeneous $NH$-bouquet of dimension $n\geq 1$ is a combinatorial $n$-ball or $n$-sphere. This follows from Theorem \ref{thm homogeneous NH-manifolds} and Corollary \ref{coro boundaryless in nhmanifold}.\end{obs}

The following result extends Lemma \ref{lemma border of q-spheres} and will be used below.

\begin{lema}\label{lemma border of q-bouquets} Let $G=B+L_1+\cdots+L_k$ be a decomposition of an $NH$-bouquet. Then every simplex in each $L_i$ is internal in $G$. Furthermore, if $\sigma\in\bd{L}_i$ then  $lk(\sigma,G)$ is an $NH$-sphere with decomposition $lk(\sigma,B)+lk(\sigma,L_i)$. In particular, $\pbd G=\pbd B-\cup_{i}L_i$.\end{lema}

\begin{proof} It is clear that every simplex internal in $L_i$ is internal in $G$. Given $\sigma\in\bd{L}_i$, by Proposition \ref{corollary homotopy type of q-bouquets} $lk(\sigma,G)=lk(\sigma,B)+ lk(\sigma,L_i)$. Also $lk(\sigma,L_i)\cap lk(\sigma,B)=\bd{lk(\sigma,L_i)}$.\end{proof}

Shellings are structure-preserving moves that transform a combinatorial manifold into another one. They were first studied by Newman \cite{New} (see also \cite{Lic,Rou,Whi}) and they turned out to be central in the theory. At the beginning of the 90's Pachner \cite{Pach} showed that two (connected) combinatorial manifolds with boundary are PL homeomorphic if and only if one can obtain one from the other by a sequence of elementary shellings, inverse shellings and simplicial isomorphisms (see also \cite{Lic}).

An \emph{elementary shelling} on a combinatorial $n$-manifold $M$ is the move $M\stackrel{sh}\rightarrow M'=\overline{M-\tau}$, where $\tau=\sigma\ast\eta$ is an $n$-simplex of $M$ with $\sigma\in \overset{\circ}{M}$ and $\partial\sigma\ast\eta\subset\partial M$. The opposite move is called an \emph{inverse shelling}. It is not hard to see that these moves are special cases of regular collapses and expansions and therefore, they preserve the structure of the manifold.

A combinatorial $n$-manifold which can be transformed into a single $n$-simplex by a sequence of elementary shellings is said to be \emph{shellable}. Shellable combinatorial $n$-manifolds are collapsible and, hence, combinatorial $n$-balls. The definition of shellability can also be extended to combinatorial $n$-spheres by declaring $S$ to be shellable if for some $n$-simplex $\sigma$, $\overline{S-\sigma}$ is a shellable $n$-ball.

The alternative, and more constructive, definition of shellability by means of inverse shellings requires the existence of a linear order $F_1,\ldots,F_t$ of all the $n$-simplices such that $F_k\cap(F_1+\cdots+F_{k-1})$ is ($n-1$)-homogeneous for all $2\leq k\leq t$. This formulation can be used to define the concept of shellability in arbitrary $n$-homogeneous complexes.
It is not difficult to see that  shellable pseudo manifolds are necessarily combinatorial balls (\cite[Proposition 4.7.22]{Bjo}). It is also known that every ball of dimension less than or equal to 2 is shellable. Examples of non-shellable $3$-balls abound in the bibliography, the first one was discovered by Furch in 1924 (see \cite{Zie} for a survey of non-shellable  $3$-balls). A way for constructing non-shellable balls for every $n\geq 3$ was presented by Lickorish in \cite{Lic2}.

Shellability in the non-homogeneous context was first considered by Bj\"orner and Wachs \cite{Bjo2} in the 90's. A finite (non-necessarily homogeneous) simplicial complex is shellable if there is a linear order $F_1,\ldots,F_t$ of its maximal simplices such that $F_k\cap(F_1+\cdots+F_{k-1})$ is ($\dim{F_k}-1$)-homogeneous for all $2\leq k\leq t$. A simplex $F_k$ is said to be a \emph{spanning simplex} if $F_k\cap(F_1+\cdots+F_{k-1})=\bd{F_k}$. It is not hard to see that the spanning simplices may be moved to any later position in the shelling order (see \cite{Koz}). It is known that a shellable complex is homotopy equivalent to a wedge of spheres, which are indexed by the spanning simplices (see \cite[Theorem 12.3]{Koz}). In particular, shellable $NH$-balls cannot have spanning simplices and shellable $NH$-spheres have exactly one spanning simplex. In general, a shellable $NH$-bouquet of index $k$ must have exactly $k$ spanning simplices.


\begin{teo}\label{thm main of shellable NH-manifolds} Let $M$ be a shellable $NH$-manifold. Then, for every shelling order $F_1,\ldots,F_t$ of $M$ and every $0\leq l\leq t$, $\mathcal{F}_l(M)=F_1+\cdots+F_l$ is an $NH$-manifold. Moreover, $\mathcal{F}_l(M)$ is an $NH$-bouquet of index $\sharp\{F_j\in\mathcal{T}\,|\,j\leq l\}$, where $\mathcal{T}$ is the set of spanning simplices. In particular, $M$ is an $NH$-bouquet of index $\sharp\mathcal{T}$.\end{teo}

\begin{proof} We proceed by induction on $n=\dim{M}$. Suppose $n\geq 1$ and fix a shelling order $F_1,\ldots,F_t$. Let $1\leq l\leq t$ and let $v\in M$ be a vertex. Since $lk(v,M)$ is a shellable $NH$-ball or $NH$-sphere with shelling order $lk(v,F_1),\ldots,lk(v,F_t)$ (some of them possibly empty), then by induction $\mathcal{F}_j(lk(v,M))$ is an $NH$-bouquet of index at most $1$ for all $1\leq j\leq l$. Since $lk(v,\mathcal{F}_l(M))=\mathcal{F}_l(lk(v,M))$ then $\mathcal{F}_l(M)$ is an $NH$-manifold. To see that $\mathcal{F}_l(M)$ is actually an $NH$-bouquet, reorder $F_1,\ldots,F_l$ so that the spanning simplices are placed at the end of the order. If $F_{p+1}$ is the first spanning simplex in the order, then $\mathcal{F}_p(M)$ is a collapsible $NH$-manifold (see \cite[Theorem 12.3]{Koz}) and hence an $NH$-ball. Then, $\mathcal{F}_l(M)=\mathcal{F}_p(M)+F_{p+1}+\cdots+F_l$ is an $NH$-bouquet of index $\sharp\{F_j\in\mathcal{T}|j\leq l\}$ by definition.
\end{proof}

\section{Regular collapses, elementary shellings and Pachner moves}

Recall that a regular expansion in an $n$-combinatorial manifold $M$ is a geometrical expansion $M\rightarrow N=M+B^n$ such that $M\cap B^n\subset\bd{M}$. As we mentioned before, this move produces a new combinatorial $n$-manifold. In this section we prove a general version of this result for $NH$-manifolds. We start with some preliminary results.

\begin{lema}\label{lemma miniball inside ball} Let $B$ be a combinatorial $n$-ball and let $L\subset \bd{B}$ be a combinatorial ($n-1$)-ball. Then, there exists a stellar subdivision $s$ such that $sB\searrow sL$.\end{lema}

\begin{proof} By \cite[Lemma III.8]{Gla} there exists a derived subdivision $\delta$ and a subdivision $\alpha$ such that $\delta B=\alpha\Delta^n$ and $\delta L=\alpha\Delta^{n-1}$, where $\Delta^{n-1}$ is an ($n-1$)-face of $\Delta^n$. Now, by \cite[Lemma III.7]{Gla} there exists a stellar subdivision $\tilde{s}$ such that $\tilde{s}\alpha\Delta^n\searrow \tilde{s}\alpha\Delta^{n-1}$ and therefore $\tilde{s}\delta B\searrow \tilde{s}\delta L$.\end{proof} 

\begin{coro}\label{coro qball inside ball} Let $B$ be a combinatorial $n$-ball and let $K\subset \bd{B}$ be a collapsible complex. Then, there exists a stellar subdivision $s$ such that $sB\searrow sK$.\end{coro}

\begin{proof} Subdivide $B$ baricentrically twice and consider a regular neighborhood $N$ of $K''$ in $\bd{B}''$ (see \cite[Corollary III.17]{Gla}). Since $K''$ is collapsible, then $N$ is an ($n-1$)-ball. Since $N\subset\bd{B}''$, by the previous lemma, there is a stellar subdivision $\tilde{s}$ such that $\tilde{s}B''\searrow \tilde{s}N$. We conclude that $\tilde{s}B''\searrow \tilde{s}N\searrow \tilde{s}K''$.\end{proof}

\begin{teo}\label{maintheorem} Let $M$ be an $NH$-manifold and $B^r$ a combinatorial $r$-ball. Suppose $M\cap B^r\subseteq\bd{B}^r$ is an $NH$-ball or an $NH$-sphere generated by ridges of $M$ or $B^r$ and that $(M\cap B^r)^{\circ}\subseteq\pbd M$. Then $M+B^r$ is an $NH$-manifold. Moreover, if $M$ is an $NH$-bouquet of index $k$ and $M\cap B^r\neq\emptyset$ for $r\neq 0$, then $M+B^r$ is an $NH$-bouquet of index $k$ (if $M\cap B^r$ is an $NH$-ball) or $k+1$ (if $M\cap B^r$ is an $NH$-sphere).\end{teo}

\begin{proof} We note first that $M,B^r\subset M+B^r$ are top generated. Since $M\cap B^r\subseteq\partial B^r$ then $B^r$ is top generated. On the other hand, if $\sigma$ is a principal simplex in $M$ which is not principal in $M+B^r$ then $\sigma$ must be in $M\cap B^r$. Since $\sigma\notin\tilde{\partial}M$ then $\sigma\notin(M\cap B^r)^{\circ}$. Hence, $\sigma$ is not principal in $M\cap B^r$, which contradicts the maximality of $\sigma$ in $M$.

We shall prove the result by induction on $r$. The case $M\cap B^r=\emptyset$ is clear, so let $r\geq 1$ and assume $M\cap B^r\neq\emptyset$. We need to prove that every vertex in $M+B^r$ is regular. It is clear that the vertices in $(M-B^r)+(B^r-M)$ are regular since $B^r$ and $M$ are $NH$-manifolds. Consider then a vertex $v\in M\cap B^r$. We claim that the pair $lk(v,M),lk(v,B^r)$ fulfills the hypotheses of the theorem. Note that $lk(v,M)$ is an $NH$-ball or $NH$-sphere,  $lk(v,B^r)$ is a combinatorial ball, since $v\in M\cap B^r\subseteq\bd{B}^r$, and $lk(v,M\cap B^r)$ is an $NH$-ball or $NH$-sphere contained in $\bd{lk(v,B^r)}$. Note also that the inclusion $(M\cap B^r)^{\circ}\subseteq\pbd M$ implies that $lk(v,M\cap B^r)^{\circ}\subseteq\pbd{lk(v,M)}$. We now check that $lk(v,M\cap B^r)$ is generated by ridges of $lk(v,M)$ or $lk(v,B^r)$. This is easily seen if $lk(v,M\cap B^r)\neq\emptyset$. For the case $lk(v,M\cap B^r)=\emptyset$ we need to show that there is a principal $0$-simplex in $lk(v,M)$ or $lk(v,B^r)$. Now, $lk(v,M\cap B^r)=\emptyset$ implies that $v$ is principal in $M\cap B^r$, so $v\in(M\cap B^r)^{\circ}\subseteq\pbd M$ and $lk(v,M)$ is an $NH$-ball (and hence, collapsible). And since $v\in M\cap B^r\subseteq\bd{B}^r$ then $lk(v,B^r)$ is a ball. Now, if $v$ is a ridge in $B^r$ then $r=1$ and, hence, $lk(v,B^1)=\ast$. If, on the other hand, $v$ is a ridge of $M$ then there exists a principal $1$-simplex $\sigma$ with $v\prec\sigma$. Since $\sigma$ is principal in $M$,  $\ast=lk(v,\sigma)$ is principal in $lk(v,M)$. Since $lk(v,M)$ is collapsible, then $lk(v,M)=\ast$. 

Therefore, by induction, $lk(v,M+B^r)$ is an $NH$-manifold. Now, if $lk(v,M\cap B^r)\neq\emptyset$, then $lk(v,M+B^r)$ is an $NH$-ball or an $NH$-sphere if $lk(v,M)$ is an $NH$-ball and it is an $NH$-sphere if $lk(v,M)$ is an $NH$-sphere. If $lk(v,M\cap B^r)=\emptyset$, we showed above that $lk(v,M)=\ast$ and $lk(v,B^r)$ is a ball or $lk(v,B^r)=\ast$ and $lk(v,M)$ is an $NH$-ball. In either case, $lk(v,M+B^r)$ is an $NH$-sphere of homotopy dimension $0$. This proves that $M+B^r$ is an $NH$-manifold.

We prove now the second part of the statement. We proceed by induction on the index $k$. Suppose first that $k=0$, i.e. $M$ is an $NH$-ball. Let $\alpha$ be a subdivision such that $\alpha M\searrow 0$, and extend $\alpha$ to all $M+B^r$. If $M\cap B^r$ is an $NH$-ball we can apply Corollary \ref{coro qball inside ball} to $\alpha(M\cap B^r)\subset \alpha \bd{B}^r$ and find a stellar subdivision $s$ such that $s\alpha B^r\searrow s\alpha(M\cap B^r)$. This implies that $s\alpha(M+B^r)\searrow s\alpha M\searrow 0$ and therefore $M+B^r$ is an $NH$-ball. If $M\cap B^r$ is an $NH$-sphere $S$ with decomposition  $S=G+L$, take any maximal simplex $\tau\in L$ with an immediate face $\sigma$ in $\bd{L}$ and consider the starring $(\tau,\hat{\tau})S$ of $S$ (see Figure \ref{fig:demo_shelling_maintheorem}). Let $\rho=\hat{\tau}\ast\sigma\in (\tau,\hat{\tau})S$. We claim that $(\tau,\hat{\tau})S-\{\rho\}$ is an $NH$-ball. On one hand, it is clear that $((\tau,\hat{\tau})S-\{\rho\})\cap\rho=\bd{\rho}$. On the other hand, $(\tau,\hat{\tau})L-\{\rho,\sigma\}$ is a combinatorial ball because it is PL-homeomorphic to $L$. Since $G$ is an $NH$-ball, $(\tau,\hat{\tau})L-\{\rho,\sigma\}$ is a combinatorial ball and $G\cap ((\tau,\hat{\tau})L-\{\rho,\sigma\})=\bd{L}-\{\sigma\}$, which is a combinatorial ball by Newman's Theorem, it follows that  $(\tau,\hat{\tau})S-\{\rho\}$ is an $NH$-ball, as claimed. Now, since $\tau\in L\subset M\cap B^r$ is principal then it must be a ridge of $M$ or of $B^r$. We analyze both cases. Suppose $\tau$ is a ridge of $B^r$ and let $\tau\prec\eta\in B^r$. Write $\eta=w\ast\tau$ (see Figure \ref{fig:demo_shelling_maintheorem}). Note that the starring $(\tau,\hat{\tau})S$ performed earlier also subdivides $\eta$ and the simplex $\rho$ lies in the boundary of $(\tau,\hat{\tau})\eta$. Consider the simplex $\nu=w\ast\rho$, which is one of the principal simplices in which $\eta$ has been subdivided. Now make the starring $(\nu,\hat{\nu})$ in $(\tau,\hat{\tau})\eta$ (see Figure \ref{fig:demo_shelling_maintheorem}). By removing the simplex $\hat{\nu}\ast\rho$ from $(\nu,\hat{\nu})(\tau,\hat{\tau})B^r$, we obtain a complex which is PL-homeomorphic to $B^r$. Then $$(\nu,\hat{\nu})(\tau,\hat{\tau})B^r-\{\hat{\nu}\ast\rho\}$$ is a combinatorial ball and it intersects $M$ in $(\tau,\hat{\tau})S-\{\rho\}$, which is an $NH$-ball. It follows that $$(\nu,\hat{\nu})(\tau,\hat{\tau})(M+B^r)-\{\hat{\nu}\ast\rho\}=(\tau,\hat{\tau})M+(\nu,\hat{\nu})(\tau,\hat{\tau})B^r-\{\hat{\nu}\ast\rho\}$$
is again an $NH$-ball. If we now plug the simplex $\hat{\nu}\ast\rho$, $(\nu,\hat{\nu})(\tau,\hat{\tau})(M+B^r)$ is an $NH$-sphere by definition. This completes the case where $\tau$ is a ridge of $B^r$. The case that $\tau$ is a ridge of $M$ is analogous.

\begin{figure}[h]
\centering
\includegraphics[width=6.00in,height=1.00in]{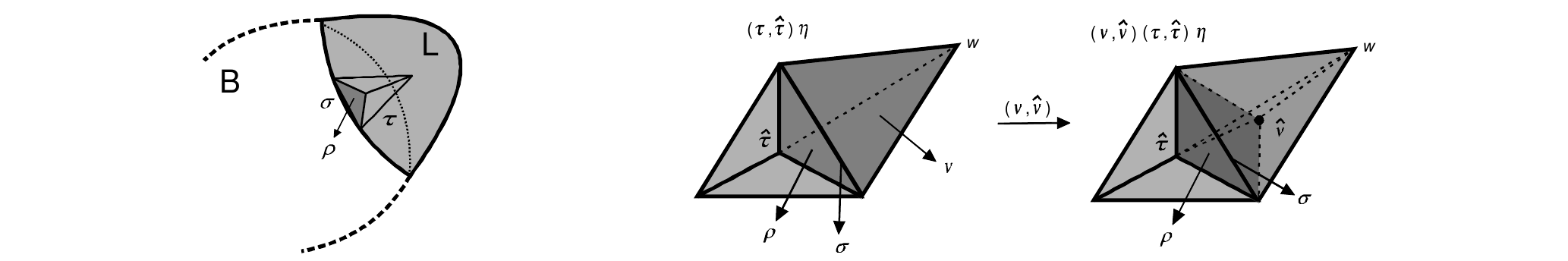}
\caption{The starrings of Theorem \ref{maintheorem}.}
\label{fig:demo_shelling_maintheorem}
\end{figure}

Suppose now that $M$ is an $NH$-bouquet of index $k\geq 1$. Write $M=G+L$ with $G$ an $NH$-bouquet of index $k-1$ and $L$ a combinatorial ball glued to $G$ along its entire boundary. If $r=0$ we obtain an $NH$-bouquet. Suppose then that  $M\cap B^r\neq\emptyset$. We claim that $B^r\cap L\subseteq\partial L$. Suppose $(L-\partial L) \cap B^r\neq\emptyset$ and let $\eta\in (L-\partial L)\cap B^r$. Now, $lk(\eta,M)=lk(\eta,L)$ is a combinatorial sphere and Corollaries \ref{coro boundaryless in nhmanifold} and \ref{coro nhbouquet cannot sphere} imply that $lk(\eta,B^r)\subset lk(\eta,M)$. But if $\tau\in B^r$ is a principal simplex containing $\eta$ then $lk(\eta,\tau)\in lk(\eta,M)$ and $\tau\in M\cap B^r\subseteq\partial B^r$, contradicting the maximality of $\tau$ in $B^r$. This proves that $B^r\cap L\subseteq\partial L$ and, therefore $M\cap B^r=G\cap B^r$. Also, $(G\cap B^r)^{\circ}\subseteq\tilde{\partial} M=\tilde{\partial} G-L\subset\tilde{\partial} G$. By induction, $G+B^r$ is an $NH$-bouquet of index $k-1$ (if $G\cap B^r=M\cap B^r$ is an $NH$-ball) or $k$ (if $G\cap B^r=M\cap B^r$ is an $NH$-sphere). In either case, $M+B^r=G+L+B^r=(G+B^r)+L$ with $(G+B^r)\cap L=G\cap L + B^r\cap L=\partial L$. Thus, $M+B^r$ is an $NH$-bouquet of index $k$ or $k+1$. This completes the proof.\end{proof}

Note that the previous theorem generalizes Alexander's Theorem on regular expansions (\cite[Theorem 3.9]{Lic}) to the non-homogeneous setting. The condition $(M\cap B)^{\circ}\subset\pbd M$ corresponds to $M\cap B\subset\bd{M}$ in the homogeneous case. We next extend the notion of regular expansion to the non-homogeneous context. This will be used to characterize the notion of shelling on $NH$-manifolds similarly as in the case of manifolds.

\begin{defi} A \emph{regular expansion} on an $NH$-manifold $M$ is a geometrical expansion $M\rightarrow M+B$ (i.e. $B$ is a ball and $M\cap B\subset\partial B$ is a ball of dimension $\dim{B}-1$) such that $(M\cap B)^{\circ}\subset\tilde{\partial}M$.\end{defi}

Recall that an inverse shelling in a combinatorial $n$-manifold $M$ corresponds to a (classical) regular expansion $M\rightarrow M+\sigma$ involving a single $n$-simplex $\sigma$. An elementary shelling is the inverse move \cite{Whi}. We investigate now shellable $NH$-balls. First we need the following result.

\begin{prop}\label{prop special inverse of maintheorem} Let $M\rightarrow M+B$ be a geometrical expansion in an $NH$-manifold $M$. If $M+B$ is an $NH$-manifold and $M,B\subset M+B$ are top generated then $(M\cap B)^{\circ}\subset\tilde{\partial}M$ (i.e. $M\rightarrow M+B$ is a regular expansion). \end{prop}

\begin{proof} Take $\rho\in(M\cap B)^{\circ}$. Since $lk(\rho,M\cap B)$ is a sphere contained in the sphere $\bd{lk(\rho,B)}$, then $lk(\rho,M\cap B)=\bd{lk(\rho,B)}$. Suppose $\rho\notin \pbd M$. Then $lk(\rho,M+B)=lk(\rho,M)+lk(\rho,B)$ is an $NH$-bouquet of index $2$ since $lk(\rho,M),lk(\rho,B)\subset lk(\rho,M+B)$ are top generated by hypothesis. This contradicts the fact that $M+B$ is an $NH$-manifold.\end{proof}

\begin{defi} Let $M$ be an $NH$-manifold. An \emph{inverse shelling} is a regular expansion $M\rightarrow M+\sigma$ where $\sigma$ is a single simplex. An \emph{elementary shelling} is the inverse move.\end{defi}

By Proposition \ref{prop special inverse of maintheorem} and Theorem \ref{thm main of shellable NH-manifolds}, we obtain the following characterization of shellable $NH$-balls in terms of elementary shellings.

\begin{coro} An $NH$-ball $B$ is shellable if and only if $B$ can be transformed into a single maximal simplex by a sequence of elementary shellings.\end{coro}

A \emph{stellar exchange} $\kappa(\sigma,\tau)$ is the move that transforms a complex $M$ into a new complex $\kappa(\sigma,\tau)M$ by replacing $st(\sigma,M)=\sigma\ast\partial\tau\ast L$ with $\partial\sigma\ast\tau\ast L$, for $\sigma\in M$ and $\tau\notin M$ (see \cite{Lic,Pach}). Note that elementary starrings and welds are particular cases of stellar exchanges (when $\tau$ or $\sigma$ is a vertex). When $L=\emptyset$ the stellar exchange is called a \emph{bistellar move}. Also, since $\kappa(\sigma,\tau)=(\tau,b)^{-1}(\sigma,a)$, two simplicial complexes are PL-homeomorphic if and only if they are related by a sequence of stellar exchanges. In the case of PL-homeomorphic combinatorial manifolds without boundary, all the moves in this sequence can be taken to be bistellar moves (see \cite{Lic,Pach} for more details). This discussion motivates the following definition.

\begin{defi} Let $M$ be a combinatorial $n$-manifold and let $\sigma\in M$ be a simplex such that $lk(\sigma,M)=\partial\tau\ast L$ with $\tau\notin M$. An \emph{$NH$-factorization} is the move $M\rightarrow M+\sigma\ast\tau\ast L$. We write $F(\sigma,\tau)M=M+\sigma\ast\tau\ast L$. When $L=\emptyset$, we call it a \emph{bistellar factorization}.\end{defi}

Note that, in fact, $NH$-factorizations can be defined for arbitrary complexes. When $\tau$ is a single vertex $b\notin M$, we will denote $M_{\sigma}^{+}=F(\sigma,b)M$. Note that $M_{\sigma}^{+}$ is  the simplicial cone of the inclusion $st(\sigma,M)\subseteq M$. Note also that, since $st(\sigma,M)$ is collapsible,  $M_{\sigma}^+\searrow M$.

By definition, the following diagram commutes (this justifies the term ``factorization").

\begin{equation*}\xymatrix{ M \ar[rr]^{\kappa(\sigma,\tau)} \ar[dr]_{F(\sigma,\tau)} & & \kappa(\sigma,\tau) M \ar[ld]^{F(\tau,\sigma)} \\
 & M+\sigma\ast\tau\ast L & }. \end{equation*}

\begin{prop}\label{ultimaprop} Let $M$ be a combinatorial $n$-manifold and let $M\longrightarrow N=F(\sigma,\tau)M$ be an $NH$-factorization. Then $N$ is an $NH$-manifold.\end{prop}

\begin{proof} Let $N=M+\sigma\ast\tau\ast L$ with $\tau\notin M$. Since $(\tau,b)N=M+b\ast\partial\tau\ast\sigma\ast L=M_{\sigma}^{+}$, by Theorem \ref{thm pl-invariancy of NH-manifolds} it suffices to prove that $M_{\sigma}^+$ is an $NH$-manifold. We prove by induction on $n$ that the simplicial cone $M_B^+$ of the inclusion of any combinatorial ball $B\subseteq M$ is an $NH$-manifold.

Denote $M_B^+=M+b\ast B$ and let $v$ be a vertex of $M_B^+$. If $v\notin B$, then $lk(v,M_B^+)=lk(v,M)$. If $v\in \overset{\circ}{B}$ then $lk(v,M_B^+)=b\ast lk(v,M)$, which is a combinatorial $n$-ball. If $v\in\partial B$ then $lk(v,M_B^+)=lk(v,M)+b\ast lk(v,B)$ is an $NH$-manifold by induction. Since $lk(v,B)$ is collapsible then $lk(v,M_B^+)\searrow lk(v,M)$, so $lk(v,M_B^+)$ is an $NH$-ball if $v\in\partial M$. If $v\notin\partial M$ then $lk(v,B)$ is strictly contained in $lk(v,M)$. It follows that there is an $n$-simplex $\eta\in M-B$ containing $v$. By Newman's Theorem, $lk(v,M)-lk(v,\eta)$ is an ($n-1$)-ball. It follows that $lk(v,M_B^+)$ is an $NH$-sphere with decomposition $$(lk(v,M-\eta)+b\ast lk(v,B))+lk(v,\eta)$$ since $lk(v,M-\eta)+b\ast lk(v,B)$ is an $NH$-ball by the previous case and $$(lk(v,M-\eta)+b\ast lk(v,B))\cap lk(v,\eta)=(lk(v,M)-lk(v,\eta))\cap lk(v,\eta)=\bd{lk(v,\eta)}.$$\end{proof}

\begin{lema}\label{lemma borde de bolas coincide} Let $M_1,M_2$ be combinatorial $n$-manifolds without boundary and let $B_i\subset M_i$ be combinatorial $n$-balls. Suppose $\overline{M_1-B_1}=\overline{M_2-B_2}$. Then, $M_1\simeq_{PL} M_2$.\end{lema}

\begin{proof} Note that $\overline{M_i-B_i}$ is a combinatorial $n$-manifold and  that $\partial B_i=\overline{M_i-B_i}\cap B_i$. Since $\partial B_2=\overline{M_2-B_2}\cap B_2=\overline{M_1-B_1}\cap B_2$ and $M_2=\overline{M_2-B_2}+B_2=\overline{M_1-B_1}+B_2$, then $B_2\cap\overline{M_1-B_1}\subseteq\partial(\overline{M_1-B_1})=\partial B_1$. Hence, $\partial B_2\subseteq \partial B_1$. Analogously, $\partial B_1\subseteq\partial B_2$. The result now follows from the fact that every ball may be starred (see \cite[Theorem II.11]{Gla}).\end{proof}

\begin{teo} Let $M,\tilde{M}$ be combinatorial $n$-manifolds (with or without boundary). If $M$ and $\tilde{M}$ are PL-homeomorphic then there exists a sequence $$M=M_1\rightarrow N_1\leftarrow M_2\rightarrow N_2\leftarrow M_3\rightarrow\cdots\leftarrow M_{r-1}\rightarrow N_{r-1}\leftarrow M_r=\tilde{M}$$ where the $N_i$'s are $NH$-manifolds, the $M_i$'s are $n$-manifolds, and $M_i, M_{i+1}\rightarrow N_i$ are $NH$-factorizations. Moreover, if $M$ and $\tilde{M}$ are closed then the converse holds. Also, in this case the $NH$-factorizations may be taken to be bistellar factorizations.\end{teo}

\begin{proof} Let $\kappa(\sigma_1,\tau_1),\ldots,\kappa(\sigma_r,\tau_r)$ be a sequence of stellar exchanges taking $M$ to $\tilde{M}$. Then for each $i$, the sequence $$M_i\stackrel{F(\sigma_i,\tau_i)}\longrightarrow N_i\stackrel{F(\tau_i,\sigma_i)}\longleftarrow M_{i+1}=\kappa(\sigma_i,\tau_i)M_i$$ is a factorization and $N_i$ is an $NH$-manifold by Lemma \ref{ultimaprop}.

For the second part of the proof, assume that $M\stackrel{F(\sigma,\tau)}\longrightarrow N\stackrel{F(\rho,\eta)}\longleftarrow\tilde{M}$ are $NH$-factorizations, with $M$ and $\tilde{M}$ $n$-manifolds and $M$ closed. Since $M+\sigma\ast\tau\ast L=\tilde{M}+\rho\ast\eta\ast T$, by a dimension argument and the homogeneity of $M$ and $\tilde M$, it follows that $\sigma\ast\tau\ast L=\rho\ast\eta\ast T$. Hence, $\overline{M-\sigma\ast\partial\tau\ast L}=\overline{N-\sigma\ast\tau\ast L}=\overline{N-\rho\ast\eta\ast T}=\overline{\tilde{M}-\rho\ast\partial\eta\ast T}$. The result now follows from Lemma \ref{lemma borde de bolas coincide}.\end{proof}

\end{document}